\documentclass[11pt]{article}
\usepackage{latexsym,amsmath,amssymb}

\newif\ifdviwin

\usepackage[english]{babel}
\usepackage{indentfirst}
\usepackage[mathscr]{eucal}
\usepackage{amssymb,amsmath,amsfonts}
\usepackage{fancybox,fancyhdr}
\usepackage{graphicx}
\usepackage[utf8]{inputenc}
\usepackage{float}
\usepackage[dvipsnames]{xcolor}
\usepackage{color}
\usepackage[pdftex]{hyperref}
\hypersetup{colorlinks=true,linkcolor=blue,citecolor=blue} 

\newif\ifdviwin

\dviwintrue

\def\m2r{\mathbb{M}^2\times\mathbb{R}}
\def\h2r{\mathbb{H}^2\times\mathbb{R}}

\let\8=\infty \let\0=\emptyset

\def\Z{\mathbb{Z}}
\def\R{\mathbb{R}}
\def\rn{\mathbb{R}^n}

\def\rnn{\mathbb{R}^{n+1}}

\def\h{\mathfrak{h}}

\def\M{\mathcal{M}}

\def\H{\mathcal{H}}

\def\S{\mathbb{S}}

\def\s2{\mathbb{S}^2}
\def\sig{\Sigma}

\def\n1{_{n+1}}
\def\nm1{_{n-1}}

\def\t1{\Theta_1}
\def\tm1{\Theta_{-1}}
\def\r3{\R^3}

\def\hi{\H\in C^1(\S^2)}

 \newtheorem{definition}{Definition}
 \newtheorem{theorem}[definition]{Theorem}
  \newtheorem{remark}[definition]{Remark}
 \newtheorem{proposition}[definition]{Proposition}

 \newenvironment{proof}{\rm \trivlist \item[\hskip \labelsep{\it
      Proof}:]}{\par\nopagebreak \hfill $\Box$ \endtrivlist}

\numberwithin{equation}{section}

\begin{document}

\mbox{}\vspace{0.4cm}

\begin{center}
\rule{15cm}{1.5pt}\vspace{0.5cm}

\renewcommand{\thefootnote}{\,}
{\Large \bf The Björling problem for prescribed mean curvature surfaces in $\r3$  \footnote{\hspace{-.75cm} Mathematics Subject
Classification: 53A10, 53C42.}}\\ \vspace{0.5cm} {\large Antonio Bueno}\\ \vspace{0.3cm} \rule{15cm}{1.5pt}
\end{center}
  \vspace{1cm}
Departamento de Geometría y Topología, Universidad de Granada,
E-18071 Granada, Spain. \\ e-mail: jabueno@ugr.es \vspace{0.3cm}

 \begin{abstract}
In this paper we prove existence and uniqueness of the Björling problem for the class of immersed surfaces in $\r3$ whose mean curvature is given as an analytic function depending on its Gauss map. As an application, we prove the existence of surfaces with the topology of a Möbius strip for an arbitrary large class of prescribed functions. In particular, we use the Björling problem to construct the first known examples of self-translating solitons of the mean curvature flow with the topology of a Möbius strip in $\mathbb{R}^3$.
 \end{abstract}

\section{Introduction}

A classical problem in minimal surface theory in $\r3$ is the \emph{Björling problem} \cite{Bjo}. This problem was posed in 1844 by Björling and asks the following:
\begin{quote}
\emph{Given a regular analytic curve $\beta(s)$ in $\r3$ and an analytic distribution of oriented planes $\Pi(s)$ along $\beta(s)$ such that $\beta'(s)\in\Pi(s)$, find \emph{all} minimal surfaces in $\r3$ containing $\beta(s)$ and such that the tangent plane distribution along $\beta(s)$ is given by $\Pi(s)$.}
\end{quote}

In 1890 Schwarz \cite{Sch} solved this problem via an integral representation formula, using holomorphic data. This problem can be applied in other generic situations: for instance, to study surfaces with certain symmetries \cite{DHKW}, and to solve global problems in the minimal surface theory \cite{AlMi,GaMi1}; see also \cite{ACG, GaMi2,GaMi3} and references therein for an outline on the development of the \emph{geometric Cauchy problem}. The Björling problem has been also studied when the mean curvature is a non-vanishing constant, see \cite{BrDo}.

Our objective in this paper is to prove existence and uniqueness of the Björling problem for a certain class of \emph{prescribed mean curvature surfaces} immersed in $\R^3$. Specifically, let be $\sig$ an oriented, immersed surface in $\r3$ and $\hi$. We say that $\sig$ has \emph{prescribed mean curvature} $\H$ (in short, $\sig$ is an $\H$-surface) if the mean curvature $H_\sig$ of $\sig$ satisfies
\begin{equation}\label{defHsup}
H_\sig(p)=\H(\eta_p),\hspace{.5cm} \forall p\in\sig,
\end{equation}
where $\eta:\sig\rightarrow\s2$ is the \emph{Gauss map} of $\sig$. 

In general, the study of hypersurfaces in $\rnn$ defined by a prescribed curvature function in terms of the Gauss map goes back, at least, to the famous Christoffel and Minkowski problems for ovaloids, see e.g. \cite{Chr}. In particular, the existence and uniqueness of ovaloids with prescribed mean curvature \eqref{defHsup} was studied among others by Alexandrov and Pogorelov \cite{Ale,Pog} but the global geometry of these surfaces remained largely unexplored. In \cite{GaMi4} the authors studied uniqueness of immersed $\H$-spheres obtaining a Hopf-type theorem for this class of immersed surfaces, and in the making they proved a standing conjecture by Alexandrov. The global properties of $\H$-hypersurfaces immersed in $\rnn$ have been recently developed in \cite{BGM}, where the authors studied several topics such as classification of rotational $\H$-hypersurfaces, a priori height estimates and a structure theorem for properly embedded $\H$-surfaces in $\r3$, and curvature estimates for stable $\H$-surfaces immersed in $\R^3$.

This paper is organized as follows: in \textbf{Section \ref{sec:resolutionbjorling}} we prove existence and uniqueness of the Björling problem for the class of analytic functions $\H\in C^\omega(\s2)$ by applying Cauchy-Kovalevskaya theorem. This theorem has been used in other works to prove existence and uniqueness of the Björling problem for minimal surfaces in three-dimensional Riemannian and Lorentzian Lie groups via a Weierstrass-type representation formula, see e.g. \cite{CMO,MMP, MeOn}.

In \textbf{Section \ref{sec:hmobius}} we restrict ourselves to the class of analytic functions $\H\in C^\omega(\s2)$ satisfying the symmetry condition $\H(-x)=-\H(x),\ \forall x\in\s2$. This condition on $\H$ ensures us that Equation \eqref{defHsup} is independent of the orientation chosen on the $\H$-surface. Bearing this in mind, we use the existence and uniqueness of the Björling problem for adequate Björling data to construct non-orientable $\H$-surfaces with the topology of a Möbius strip. A particular analytic function with this symmetry condition is the one given by $\H(x)=\langle x,e_3\rangle,\ \forall x\in\s2$. The $\H$-surfaces arising for this prescribed function are the \emph{self-translating solitons of the mean curvature flow} in $\r3$, a well studied class of surfaces in the past decades. See e.g. \cite{CSS,Hui,MSHS} and references therein for relevant works regarding this topic. As an application, we construct self-translating solitons of the mean curvature flow in $\r3$ with the topology of a \emph{Möbius strip}. After a detailed search in the literature, we can assert that these \emph{translating Möbius strips} are the first known examples of self-translating solitons with non-orientable topology.

Finally, in \textbf{Section \ref{sec:masejemplos}} we use the solution of the Björling problem to construct further examples of $\H$-surfaces for analytic prescribed functions. In particular, we construct \emph{helicoidal} $\H$-surfaces, and $\H$-surfaces similar to Enneper's classical minimal surface.
\vspace{.5cm}

\textbf{Acknowledgment.} The author wants to express his gratitude to his Ph.D. advisor Pablo Mira, for fruitful conversations about this topic.

\section{The Björling problem for $\H$-surfaces}\label{sec:resolutionbjorling}
In this paper we will denote by $C^\omega(\s2)$ to the class of \emph{analytic functions} defined on the sphere $\s2$ in $\r3$.

\begin{definition}
Let be $I\subset\R$ an open interval. A pair of \emph{Björling data} for $\H$-surfaces in $\r3$ is a regular analytic curve $\beta:I\rightarrow\r3$ and an analytic vector field $B:I\rightarrow\r3$ along $\beta(s)$ such that $|\beta'(s)|-|B(s)|=\langle\beta'(s),B(s)\rangle=0,\ \forall s\in I$.
\end{definition}

From this definition, we get two obvious consequences: First, given $\H\in C^\omega(\s2)$, a regular analytic curve $\beta(s)$ and an oriented distribution of planes $\Pi(s)$ along $\beta(s)$, we get a pair of Björling data by just defining $B(s)=J\beta'(s)$, where $J$ denotes the $\pi/2$-rotation in the tangent plane $\Pi(s)$. And second, let us denote by $\nabla$ to the Riemannian connection of $\sig$ and suppose that $\beta(s)$ is parametrized by arc-length. Moreover, suppose that $\beta(s)$ is not a straight line. Then there exists $s_0\in\R$ such that $|\beta''(s_0)|\neq 0$. Let $I_0$ be the largest subinterval of $I$ containing $s_0$ such that $|\beta''(s)|\neq 0$ for all $s\in I_0$. If we define $V(s):=\nabla_{\beta'(s)}\beta'(s)$ then $B(s)=V(s)/|V(s)|$ is an analytic unit vector field along $\beta(s)$ such that $\langle\beta'(s),B(s)\rangle=0$. Thus, $B(s)$ determines an oriented distribution of planes $\Pi(s)$ along $\beta(s)$ by defining $\Pi(s)=\beta(s)+B(s)^\bot$. In particular, the Björling problem generalizes the problem of finding a surface which contains a given curve as a geodesic.

Although we do not have a Weierstrass representation for $\H$-surfaces, and thus we cannot solve the Björling problem with an \emph{explicit} integral representation formula just as in the case $\H=0$ (see e.g. \cite{Mir}), we can prove existence and uniqueness of this problem by applying different methods, as other have done in similar situations; see e.g. \cite{CMO, MeOn}.

Let $\sig$ be an orientable Riemannian surface and let $\psi:\sig\rightarrow\r3$ be an isometric immersion of $\sig$ in $\r3$. Then, it is well known that the coordinates of $\psi$ satisfy the elliptic PDE
\begin{equation}\label{laplainmersion}
\Delta_\sig\psi=2H_\sig\eta,
\end{equation}
where $\Delta_\sig$ stands for the Laplace-Beltrami operator on $\sig$, $\eta$ denotes the Gauss map of $\sig$ and $H_\sig$ is the mean curvature of $\sig$ computed with respect to $\eta$. 

Recall that $\sig$ inherits a Riemann surface structure induced by its first fundamental form, and thus we can consider a conformal coordinate $z=s+it$ defined in a simply connected domain $\Omega\subset\mathbb{C}$, and we define the usual \emph{Wirtinger operators} $\partial_z=1/2(\partial_s-i\partial_t)$, $\partial_{\overline{z}}=1/2(\partial_s+i\partial_t)$. Then, the induced metric on $\sig$ is expressed as $\langle\cdot,\cdot\rangle=\lambda^2|dz|^2$, where $|dz|^2$ is the flat metric on $\Omega$ and $\lambda^2=\langle\partial_s,\partial_s\rangle=\langle\partial_t,\partial_t\rangle>0$ is the \emph{conformal factor}. For such a conformal coordinate, the operator $\Delta_\sig$ writes as
\begin{equation}\label{laplainmersion2}
\Delta_\sig=\frac{1}{\lambda^2}\Delta_0=\frac{1}{\lambda^2}(\partial_{ss}+\partial_{tt})=\frac{4}{\lambda^2}\partial_{z\overline{z}},
\end{equation} 
where $\Delta_0$ denotes the usual flat Laplacian, and we used the relation of the Laplace-Beltrami operator between two conformal metrics. On the other hand, the Gauss map $\eta$ of $\sig$ has the following expression with respect to $z$:
\begin{equation}\label{normalconforme}
\eta=\frac{2}{i}\frac{\psi_z\wedge\psi_{\overline{z}}}{\sqrt{\langle\psi_z\wedge\psi_{\overline{z}},\psi_z\wedge\psi_{\overline{z}}\rangle}}=\frac{2}{i\lambda^2}\psi_z\wedge\psi_{\overline{z}}.
\end{equation}

Suppose now that the immersion $\psi:\sig\rightarrow\r3$ defines an $\H$-surface for some $\H\in C^\omega(\s2)$. By Equations \eqref{laplainmersion2} and \eqref{normalconforme}, Equation \eqref{laplainmersion} writes as
\begin{equation}\label{pdeinmersion}
\psi_{z\overline{z}}=-i\H(\eta)\psi_z\wedge\psi_{\overline{z}}.
\end{equation}
Viewing the immersion in coordinates $\psi=(\psi_1,\psi_2,\psi_3)$, then Equation \eqref{pdeinmersion} can be seen as a system of partial differential equations. In this setting, we can prove existence and uniqueness of the Björling problem for the class of $\H$-surfaces, as stated next:

\begin{theorem}[\textbf{Björling problem for $\H$-surfaces}]\label{thbjorling}
Let be $\H\in C^\omega(\s2)$ and $\beta(s),B(s)$ a pair of Björling data defined on a real interval $I$. Then, there exists an open domain $\Omega\subset\mathbb{C}$ containing $I\times\{0\}$ and a conformal immersion $\psi:\Omega\rightarrow\r3$ that solves the following system
\begin{equation}\label{ecubjorling}
\left\lbrace\begin{array}{l}
\psi_{z\overline{z}}=-i\H(\eta)\psi_z\wedge\psi_{\overline{z}},\\
\psi(s,0)=\beta(s),\\
\psi_t(s,0)=B(s).
\end{array}\right.
\end{equation}

As a consequence, $\psi$ defines an $\H$-surface $\sig$ that contains the curve $\beta(s)$, and the tangent plane distribution $T_{\beta(s)}\sig$ at each point $\beta(s)\in\sig$ is spanned by the vectors $\beta'(s)$ and $B(s)$.
\end{theorem}

\begin{proof}
The system \eqref{ecubjorling} is elliptic and of \emph{Cauchy-Kovalevskaya} type, and thus it has local existence and uniqueness; see \cite{Pet} for a proof of Cauchy-Kovalevskaya theorem. As the system \eqref{ecubjorling} is elliptic without characteristic points, we have that the existence and uniqueness extends to the whole interval $I$ where $\beta(s)$ and $B(s)$ are defined. Thus, there exist $\delta>0$ and functions $(\psi_1,\psi_2,\psi_3)$ defined in $I\times (-\delta,\delta)\subset\mathbb{C}$ such that $\psi=(\psi_1,\psi_2,\psi_3)$ is a solution of \eqref{pdeinmersion} satisfying $\psi(s,0)=\beta(s),\ \psi_t(s,0)=B(s),\ \forall s\in I$. 

First, observe that equation $\langle\psi_{z\overline{z}},\psi_z\rangle=0$ holds by substituting $\psi_{z\overline{z}}$ for its expression \eqref{pdeinmersion}, and thus the function $\langle\psi_z,\psi_z\rangle$ is holomophic. As $\psi$ satisfies the described initial conditions, the function $\langle\psi_z,\psi_z\rangle$ evaluated at $(s,0)$ is equal to $|\beta'(s)|^2-|B(s)|^2-2i\langle\beta'(s),B(s)\rangle$. This expression vanishes identically, as $\beta'(s)$ and $B(s)$ are orthogonal vector fields with the same length. In this conditions, $\langle\psi_z,\psi_z\rangle$ is a holomorphic function vanishing at the real axis and by the identity principle of holomorphic functions, $\langle\psi_z,\psi_z\rangle$ is identically zero in $I\times(-\delta,\delta)$. We conclude that the map $\psi:I\times(-\delta,\delta)\rightarrow\r3$ is conformal.

By Equation \eqref{ecubjorling}, the regularity of $\beta(s)$, and the orthogonality of $\beta'(s)$ and $B(s)$, it is clear that $\psi$ defines an immersion on an open set $\Omega\subset\mathbb{C}$ containing $I\times\{0\}$. Thus, $\psi$ is a conformal immersion of an $\H$-surface.

%Now we prove that the map $\psi$ defines an immersion. This property is given by the condition $\langle\psi_z,\psi_{\overline{z}}\rangle>0$. If we evaluate again the function $\langle\psi_z,\psi_{\overline{z}}\rangle$ at the real axis $(s,0)$, we get $1/4(|\beta'(s)|^2+|B(s)|^2)>0$. Consequently, $\langle\psi_z,\psi_{\overline{z}}\rangle>0$ in the real axis, and thus there exists an open set $\Omega\subset I\times(-\delta,\delta)$ such that the map $\psi:\Omega\rightarrow\R$ is a conformal immersion. Lastly, as $\psi$ is a solution of Equation \eqref{pdeinmersion}, $\psi$ defines a conformal immersion of an $\H$-surface. 

This concludes the proof of Theorem \ref{thbjorling}.
\end{proof}

\begin{remark}
As we pointed out, in general system \eqref{ecubjorling} cannot be explicitly integrable, as happens when $\H=0$. However, we can numerically solve it for producing images of $\H$-surfaces in $\R^3$. Indeed, let us denote by $\psi(s,t)=(\psi_1(s,t),\psi_2(s,t),\psi_3(s,t))$. Given $\H\in C^\omega(\s2)$, $\beta(s),B(s):[s_0,s_1]\rightarrow\r3$ a pair of Björling data and $\delta>0$, the solution of the Björling problem can be plotted using standard software with the command
\end{remark}
\vspace{-.3cm}

\texttt{ParametricPlot3D\textcolor{blue}{[}\\
Evaluate\textcolor{orange}{[}\\
First\textcolor{red}{[}$\{\psi[s,t]\}$/.\\
NDSolve\textcolor{green}{[}$\textcolor{purple}{\{}$\\
$D[\psi[s,t],s,s]+D[\psi[s,t],t,t]==2\H[\eta[s,t]]D[\psi[s,t],s]\wedge D[\psi[s,t],t]$,\\
Thread[$\psi[s,0]==\beta[s],D[\psi[s,t],t][s,0]==B[s]$]$\textcolor{purple}{\}}$,\\
$\{\psi_1,\psi_2,\psi_3\},\{s,s_0,s_1\},\{t,-\delta,\delta\}$\textcolor{green}{]}\textcolor{red}{]}\textcolor{orange}{]},$\{s,s_0,s_1\},\{t,-\delta,\delta\}$\textcolor{blue}{]}}.
\vspace{.3cm}

For example, consider the analytic function $\H(x)=\langle x,e_3\rangle,\ \forall x\in\s2$. The $\H$-surfaces arising for this prescribed function are the \emph{self-translating solitons of the mean curvature flow}. The rotational self-translating solitons of the mean curvature flow are classified as follows: an entire, strictly convex vertical graph  that intersects the axis of rotation orthogonally, called the \emph{bowl soliton}; and a one parameter family of properly embedded annuli, with both ends pointing towards the $e_3$ direction, called the \emph{wing-like solitons} or \emph{translating catenoids}. The family of wing-like solitons are parametrized by the \emph{neck size}, i.e. the minimum distance to the axis of rotation, attained at a circumference of radius equal to the \emph{waist}, see \cite{CSS} for details.

Bearing this in mind, we can recover this family by choosing adequate Björling data. Indeed, consider the one parameter family of Björling data $\beta_\tau(s)=\tau(\cos s,\sin s,0)$ and $B_\tau(s)=(0,0,\tau),\ \forall s\in\R,\ \tau>0$. Then, for each fixed $\tau>0$, the translating soliton $\sig_\tau$ generated by this Björling data corresponds to the wing-like soliton with neck size equal to $\tau$. When $\tau\rightarrow 0$, the sequence $\sig_\tau$ converges smoothly to a double cover of the bowl soliton minus the vertex. See Figure \ref{fig:hcatbjorling} for a plot of the wing-like soliton with neck size equal to 1.

\begin{figure}[h]
\centering
\includegraphics[width=.75\textwidth]{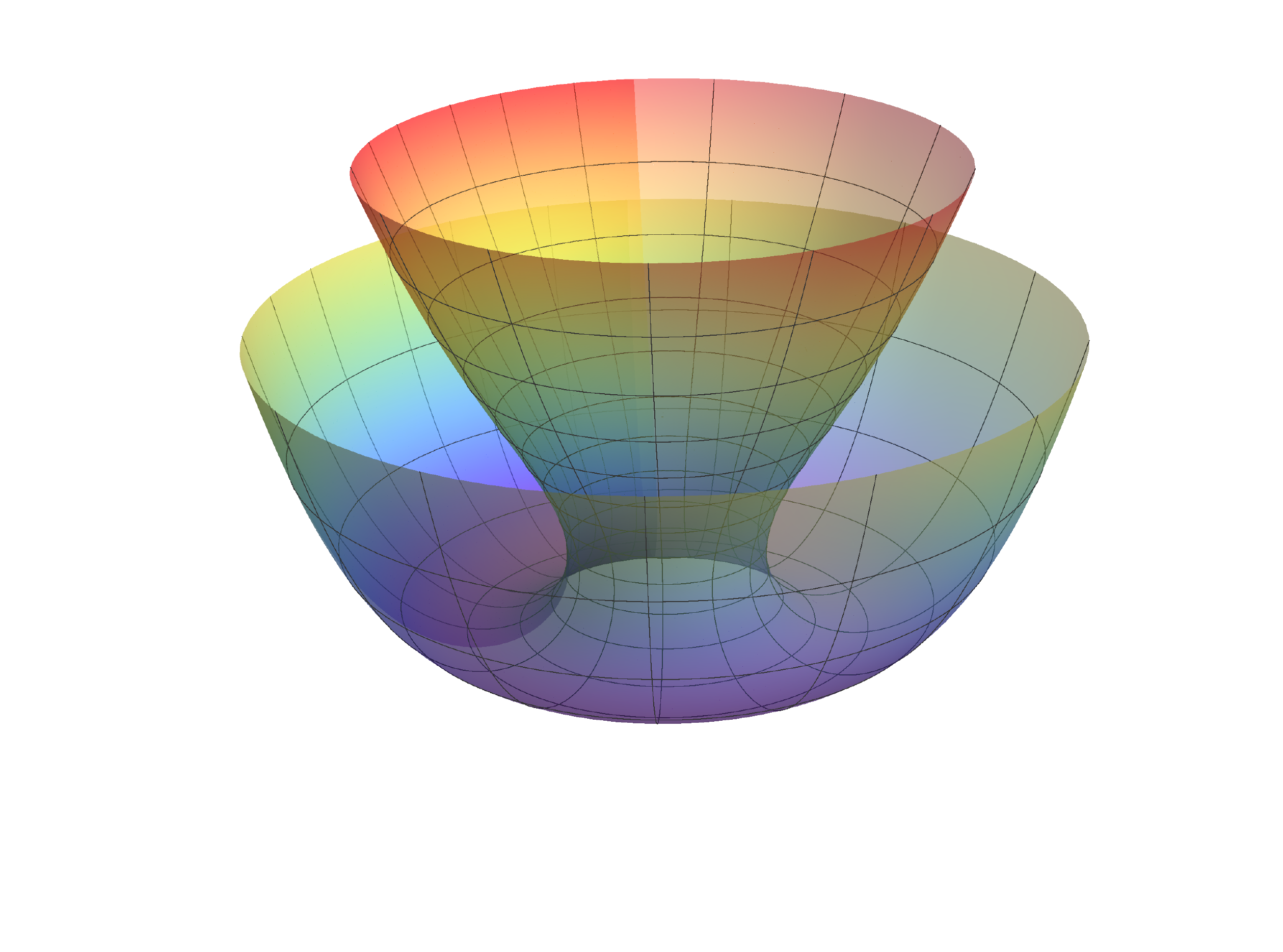}
\vspace{-1.5cm}
\caption{A wing-like soliton for the Björling data $\beta(s)=(\cos s,\sin s,0)$ and $B(s)=(0,0,1),\ s\in\R$.}
\label{fig:hcatbjorling}
\end{figure}

\section{$\H$-Möbius strips in $\r3$}\label{sec:hmobius}

The Björling problem adds a large amount of $\H$-surfaces to the ones studied in \cite{BGM}, for any given $\H\in C^\omega(\s2)$. By choosing adequate Björling data, the examples arising may have some prescribed symmetries, as well as some topological properties. In this context, the Björling problem motivates us for the search of \emph{non-orientable} surfaces. However, we must recall that our $\H$-surfaces are supposed to be oriented, in virtue of Definition \ref{defHsup}, and thus the concept of non-orientable $\H$-surface makes no sense for an arbitrary function $\H$. For instance, if $\H$ is chosen as a positive constant, examples with non-orientable topology do not exist. Thus, some condition on the prescribed function $\H\in C^\omega(\s2)$ must be imposed in order to make Definition \ref{defHsup} independent on the orientation chosen on the $\H$-surface. Bearing this in mind, the mildest hypothesis is that $\H$ has to be \emph{antipodally antisymmetric}, that is $\H(-x)=-\H(x)$ for all $x\in\s2$.

Going back to the propose of studying non-orientable $\H$-surfaces, the most recurrent examples are the surfaces with the topology of a \emph{Möbius strip}; we advise \cite{Mee,Mir} as two relevant works regarding minimal Möbius strips in $\rn$ given as the solution of the Björling problem. The following result is inspired in the ideas developed in \cite{Mir}.

\begin{proposition}\label{hcintas}
Let $\H\in C^\omega(\s2)$ such that $\H(-x)=-\H(x),\ \forall x\in\s2$, and let $\beta(s),B(s)$ be Björling data such that $\beta(s)$ is $T$-periodic and $B(s)$ is $T$-antiperiodic, i.e. $B(s+T)=-B(s)$, for some $T>0$. 

Then, the $\H$-surface generated by Theorem \ref{thbjorling} for this Björling data has the topology of a Möbius strip, with fundamental group generated by $\beta(s)$.

Conversely, every $\H$-\emph{Möbius strip} is generated in this way.
\end{proposition}

\begin{proof}
We will give a sketch of the proof, since it is an adaptation of the one used to prove Lemma 3 in \cite{Mir}.

First, let $\psi:\M\rightarrow\r3$ be an immersion of an $\H$-Möbius strip, and let $\Gamma(s)$ be a regular, analytic, closed curve in $\M$ that generates its fundamental group. As $\Gamma(s)$ is closed, it is $T$-periodic for some $T>0$. Denote by $\widetilde{\M}$ to its \emph{two sheeted cover}, where we have defined an antiholomorphic involution $I$ without fixed points, and let $\pi:\widetilde{\M}\rightarrow\R^3$ be the canonical projection. In this setting, there exists a regular, analytic, closed curve $\widetilde{\Gamma}(s)$ that generates the fundamental group of $\widetilde{\M}$, defined by $\Gamma(s)=\pi(\widetilde{\Gamma}(s))$, and in particular it is $2T$-periodic.

If we consider $\widetilde{\psi}:\widetilde{\M}\rightarrow\R^3$ given by $\widetilde{\psi}=\psi\circ\pi$, then $\beta(s)=\widetilde{\psi}(\widetilde{\Gamma}(s))$ is a regular, analytic $T$-periodic curve in $\R^3$. Denoting by $\Pi(s)$ the oriented tangent plane distribution of $\widetilde{\psi}$ along $\beta(s)$ we have that $\Pi(s+T)$ and $\Pi(s)$ agree with opposite orientation. Therefore, if $B(s):=J\beta'(s)$ is the $\pi/2$ rotation of $\beta'(s)$ in $\Pi(s)$, it happens that $B(s+T)=-B(s)$. Thus, we have proved that every $\H$-Möbius strip defines a pair of Björling data with the periodicity properties stated.

%Let $\widetilde{\M}$ be the \emph{two sheeted cover} of $\M$ and denote by $\pi:\widetilde{\M}\rightarrow\M$ the canonical projection. In this setting, $\widetilde{\M}$ inherits a Riemann surface structure, endowed with an antiholomorphic involution without fixed points, i.e. a map $I:\widetilde{\M}\rightarrow\widetilde{\M}$ such that $I(z_1)=z_2$ if and only if $z_1\neq z_2$ and $\pi(z_1)=\pi(z_2)$. Note that $\widetilde{\M}$ is topologically a cylinder, and so there exists a regular analytic closed curve $\widetilde{\Gamma}$ that generates the fundamental group of $\widetilde{\M}$, defined as $\pi(\widetilde{\Gamma})=\Gamma$. As $\widetilde{\Gamma}$ is a closed curve, we can parametrize $\widetilde{\Gamma}$ as $\widetilde{\Gamma}(s):\R\rightarrow\widetilde{\Gamma}$, where $\widetilde{\Gamma}(s)$ is $2T$-periodic for some $T>0$, and satisfies $I(\widetilde{\Gamma}(s))=\widetilde{\Gamma}(s+T),\ \forall s\in\R$. In particular, $\Gamma(s):=\pi(\widetilde{\Gamma}(s))$ is $T$-periodic.
%
%Define next $\widetilde{\psi}:\widetilde{\M}\rightarrow\r3$ by the relation $\widetilde{\psi}=\psi\circ\pi$. Then, $\beta(s)=\widetilde{\psi}(\widetilde{\Gamma}(s))$ is a regular analytic $T$-periodic curve in $\r3$. Besides, denoting by $\Pi(s)$ the oriented tangent planes distribution of $\widetilde{\psi}$ along $\beta(s)$, we have that the planes $\Pi(s+T)$ and $\Pi(s)$ agree with opposite orientation. If we denote by $J$ the $\pi/2$-rotation in each oriented plane, then the field $B(s):=J\beta'(s)$ satisfies $B(s+T)=-B(s),\ \forall s\in\R$. That is, $B(s)$ is $T$-antiperiodic. 

\begin{figure}[h]
\centering
\includegraphics[width=.8\textwidth]{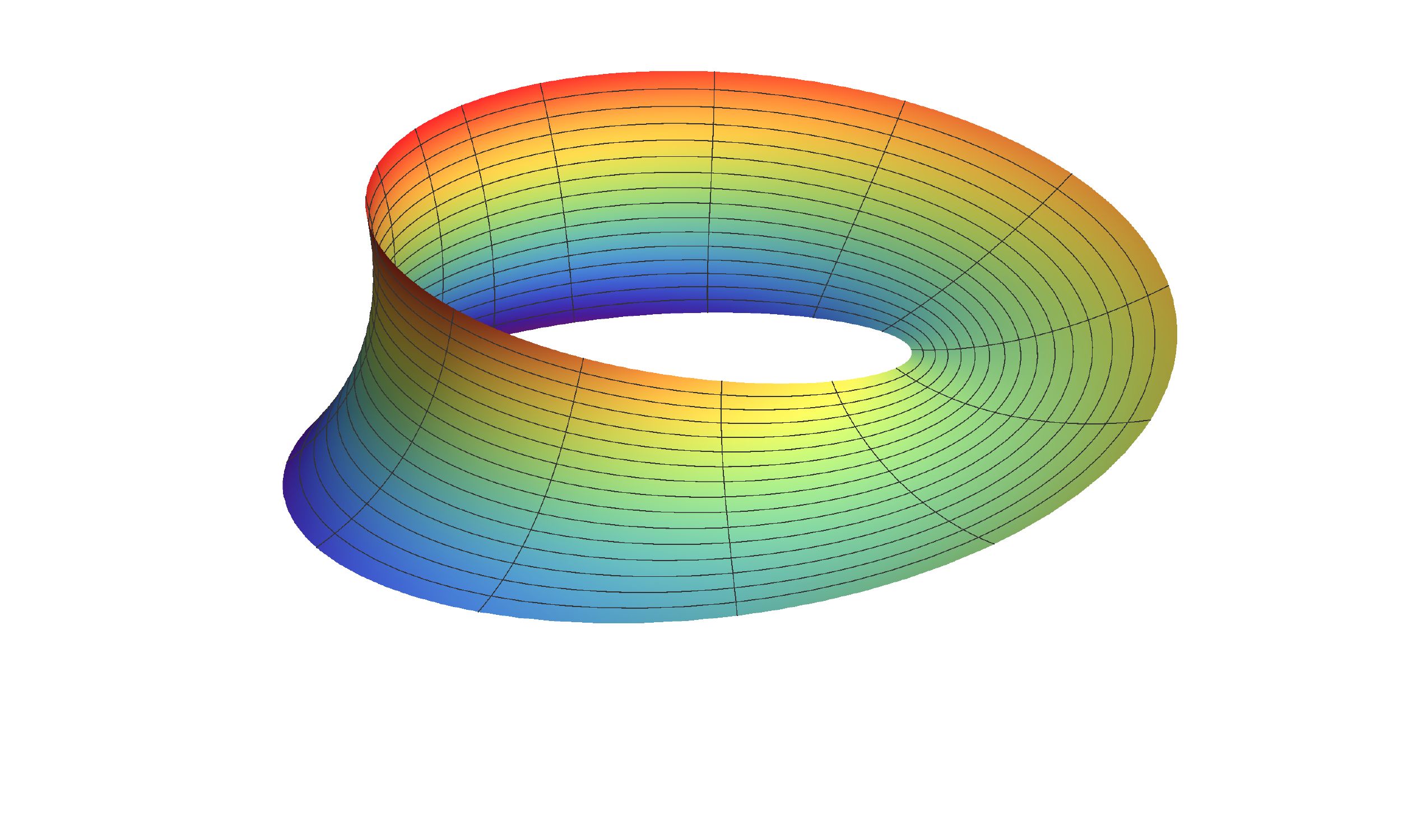}
\vspace{-1cm}
\caption{A self-translating soliton of the mean curvature flow with the topology of a Möbius stirp, generated by the Björling data $\beta(s)=1/2(\cos 2s,\sin 2s,0),\ B(s)=\cos s(\cos s,\sin s,0)+\sin s(0,0,1),\ \forall s\in\R$.}
\label{fig:hmobius}
\end{figure}

Conversely, let $\beta(s)$ and $B(s)$ be a pair of Björling data such that $\beta(s)$ is $T$-periodic and $B(s)$ is $T$-antiperiodic. In particular, they are defined on the entire real line. Let $\psi:\Omega\rightarrow\R$ be the solution of this Björling problem given by Theorem \ref{thbjorling}. The $2T$-periodicity of $\beta(s)$ and $B(s)$ along with the uniqueness of the Björling problem, ensures us that $\psi$ is well defined on the quotient $\Omega/(2T\mathbb{Z})$, which is a topological cylinder.

We can suppose that $\Omega$ is symmetric with respect to the conjugation, i.e. $\overline{z}\in\Omega,\ \forall z\in\Omega$, and thus we can define on $\Omega/(2T\mathbb{Z})$ the map
\begin{equation}
\begin{array}{cccc}
I&:\Omega/(2T\mathbb{Z})&\longrightarrow&\Omega/(2T\mathbb{Z})\\
& z&\longmapsto& \overline{z}+T.
\end{array}
\end{equation}

It is also clear that $I$ is an antiholomorphic involution without fixed points, and thus it reverses the orientation of $\Omega/(2T\mathbb{Z})$. Moreover, $I$ defines the following equivalence relation on $\Omega/(2T\mathbb{Z})$: two points $z,w\in\Omega/(2T\mathbb{Z})$ are related if and only if $I(z)=w$. In this situation, the cylinder $\Omega/(2T\mathbb{Z})$ is the orientable two sheeted cover of the space $(\Omega/(2T\mathbb{Z}),I)$, with the canonical projection $\mathfrak{p}:\Omega/(2T\mathbb{Z})\rightarrow(\Omega/(2T\mathbb{Z}),I)$.

Because $\psi$ is an immersion, the unitary vector field $\eta:\Omega\rightarrow\r3$ as defined in \eqref{normalconforme} is a well defined, unitary, normal vector field for $\psi$. Given $z,w\in\Omega/(2T\Z)$ such that $w=I(z)$, the uniqueness of the Björling problem implies that the unit normals $\eta(z)$ and $\eta(w)$ are opposite. Bearing this in mind, we have:
\begin{equation}\label{pegan}
H_{\psi}(\psi(w))=\H(\eta(w))=\H(\eta(I(z)))=\H(-\eta(z))=-\H(\eta(z))=-H_{\psi}(\psi(z)),
\end{equation}
and thus the mean curvature at the point $\psi(z)$ has opposite sign to the mean curvature at the point $\psi(I(z))$, for all $z\in\Omega$.

Again, the uniqueness of Theorem \ref{thbjorling} allows us to conclude that the quotient map $\widetilde{\psi}(z)=(\psi\circ\mathfrak{p})(z)$ for all $z\in\Omega/(2T\Z)$ is a well defined, conformal immersion of an $\H$-surface in $\r3$ having the topology of a Möbius strip, and in particular is non-orientable. This concludes the proof of Proposition \ref{hcintas}.
\end{proof}

Note that the function $\H(x)=\langle x,e_3\rangle,\ \forall x\in\s2$ lies in the hypothesis of Proposition \ref{hcintas}. Thus, we ensure the existence of self-translating solitons of the mean curvature flow with the topology of a Möbius strip, which we will refer to as \emph{translating Möbius strips}, see Figure \ref{fig:hmobius}. After a detailed search in the literature, we can assert that this construction gives the first example of a self-translating soliton of the mean curvature flow with non-orientable topology.

In Figure \ref{fig:noorientable} we show the construction of a non-orientable translating soliton constructed by half-rotating the vector field $B(s)$ 7 times along the curve $\beta(s)$ before it closes. This surface is homeomorphic to the Möbius strip showed in Figure \ref{fig:hmobius}. When the mean curvature vanishes, the minimal surfaces as the one appearing in Figure \ref{fig:noorientable} were firstly constructed in \cite{Mir}, see also the independent work of \cite{MeWe}. These surfaces are commonly known as \emph{bended helicoids}.

\begin{figure}[H]
\centering
\includegraphics[width=.65\textwidth]{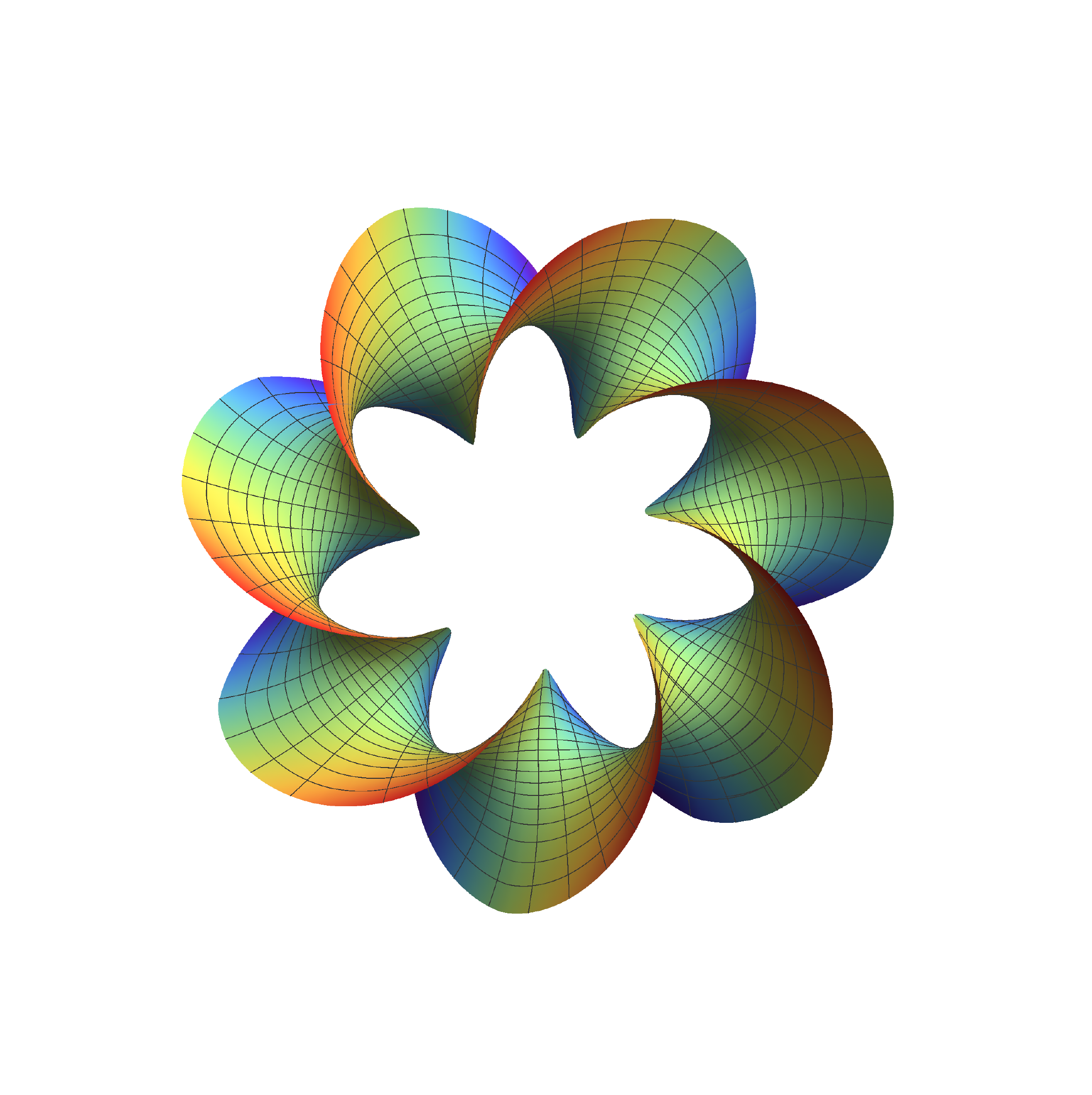}
\vspace{-1cm}
\caption{An $\H$-bended helicoid with the topology of a Möbius strip for the analytic function $\H(x)=\langle x,e_3\rangle,\ x\in\s2$.}
\label{fig:noorientable}
\end{figure}

\section{Further examples of $\H$-surfaces via the Björling problem}\label{sec:masejemplos}

In this last Section we show the existence of some examples of $\H$-surfaces immersed in $\r3$, motivated by the analogous examples defined in the minimal surface theory. The main difference with the minimal case $\H=0$ in $\r3$ is that we fail to have a Weierstrass representation, and thus explicit parametrizations of these surfaces are not expected. Even so, we can prove existence and uniqueness by means of Theorem \ref{thbjorling} for adequate Björling data with some prescribed symmetries, and obtain $\H$-surfaces that are the analogous to the famous examples in the minimal surface theory.
\vspace{.5cm}

\textbf{\underline{$\H$-Helicoids}}

We choose as Björling data the vertical curve $\beta(s)=(0,0,s)$ and a $T$-periodic, analytic, unitary vector field $B(s)$ along $\beta(s)$, for some $T>0$, and let $\sig$ be $\H$-surface given as the solution of the Björling problem for this Björling data.

The unit normal vector field at the $e_3$-axis, namely $\eta(s):=\beta'(s)\wedge B(s)$, is a horizontal vector field satisfying $\eta(s)=\eta(s+T),\ \forall s\in\R$, and the Björling data $\beta(s+T),B(s+T)$ agree with the Björling data $\beta(s)+T e_3,B(s)$. Moreover, as the condition $\eta(s)=\eta(s+T),\ \forall s\in\R$ holds, the points $\beta(s)$ and $\beta(s+T)$ have the same mean curvature.

Bearing this in mind, if we denote by $\Theta(p)=p+T e_3,\ \forall p\in\r3$ , the uniqueness of the Björling problem ensures us that $\sig$ is invariant by the discrete group of translations $T\mathbb{Z}\Theta$ in the $e_3$-direction. Moreover, starting at some $s_0\in\R$, $\sig$ \emph{twists} jointly with $\eta(s)$ around the $e_3$-axis until reaching the instant $\eta(s_0+T)=\eta(s_0)$, generating a simply connected \emph{fundamental part} of $\sig$. Repeating this process, which is just translating this fundamental part of $\sig$ under the action of $\Theta$, we get the whole $\H$-surface $\sig$.

We will refer to these $\H$-surfaces as $\H$-helicoids, since they generalize the usual minimal helicoids in $\r3$. See Figure \ref{fig:hhelicoide} for a plot of an $\H$-helicoid for the particular function $\H(x)=\langle x,e_3\rangle$, and the Björling data $\beta(s)=(0,0,s),\ B(s)=(\cos s,\sin s,0)$.

\begin{figure}[H]
\centering
\includegraphics[width=.25\textwidth]{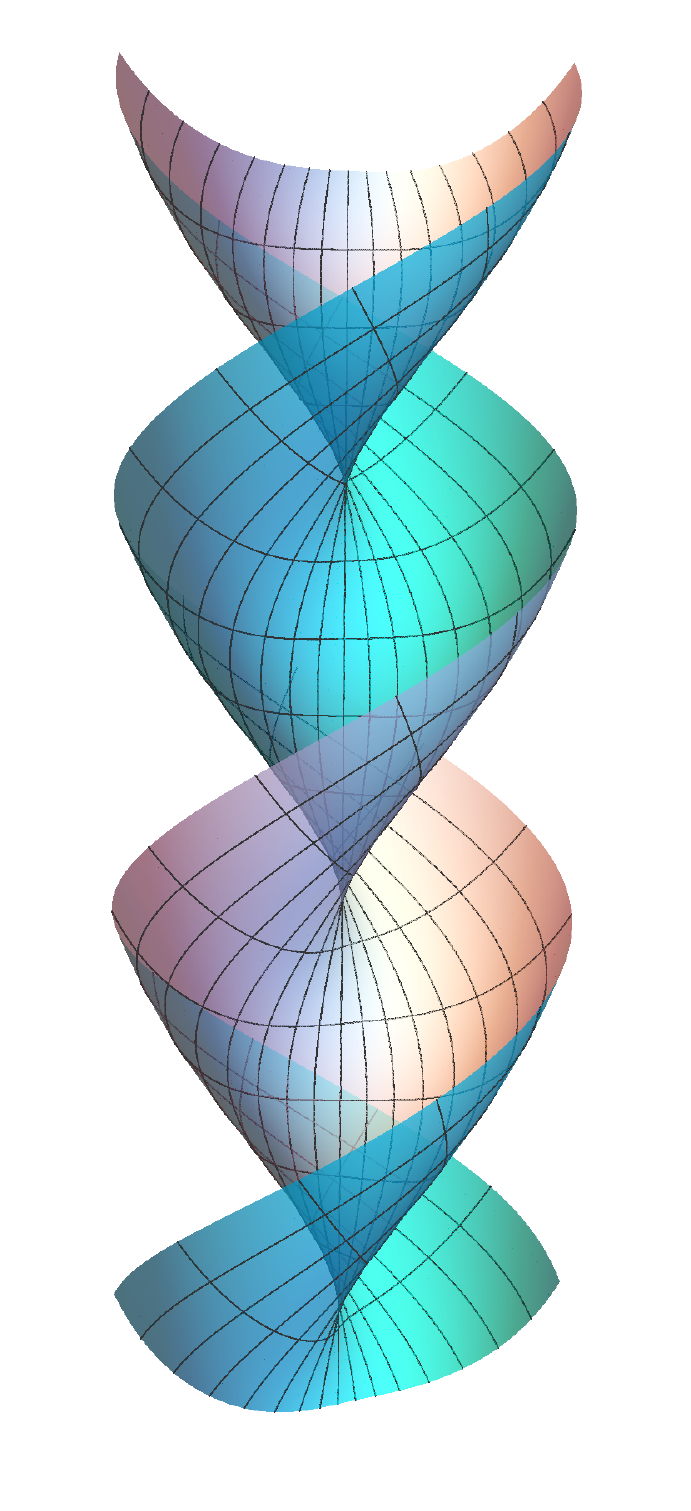}
\caption{An $\H$-helicoid for the analytic function $\H(x)=\langle x,e_3\rangle,\ x\in\s2$.}
\label{fig:hhelicoide}
\end{figure}
\vspace{.5cm}

\textbf{\underline{Enneper-type $\H$-surfaces}}

Finally, we construct $\H$-surfaces based on some curves contained in the well-known Enneper's surface, one of the most famous examples in the minimal surface theory.

First, consider the curve
$$
\beta(s)=(0,-s(1-s^2/3)/3,-s^2/3)
$$
and the vector field
$$
B(s)=1/3(1+s^2)(1,0,0).
$$
Then, both $\beta(s)$ and $B(s)$ are analytic and satisfy $|\beta'(s)|-|B(s)|=\langle\beta'(s),B(s)\rangle=0,\ \forall s\in\R$, i.e. they can be chosen to be Björling data. If $\H=0$, then the surface given by Theorem \ref{thbjorling} is Enneper's minimal surface. In particular, the curve $\beta(s)$ is obtained as intersecting Enneper's minimal surface with the plane $\{x=0\}$, which is a plane of reflection symmetry of Enneper's minimal surface. For the function $\H(x)=\langle x,e_3\rangle,\ \forall x\in\s2$, the $\H$-surface arising is a translating soliton of the mean curvature flow that resembles indeed to Enneper's minimal surface, see Figure \ref{fig:henneper}.

\begin{figure}[H]
\centering
\includegraphics[width=.6\textwidth]{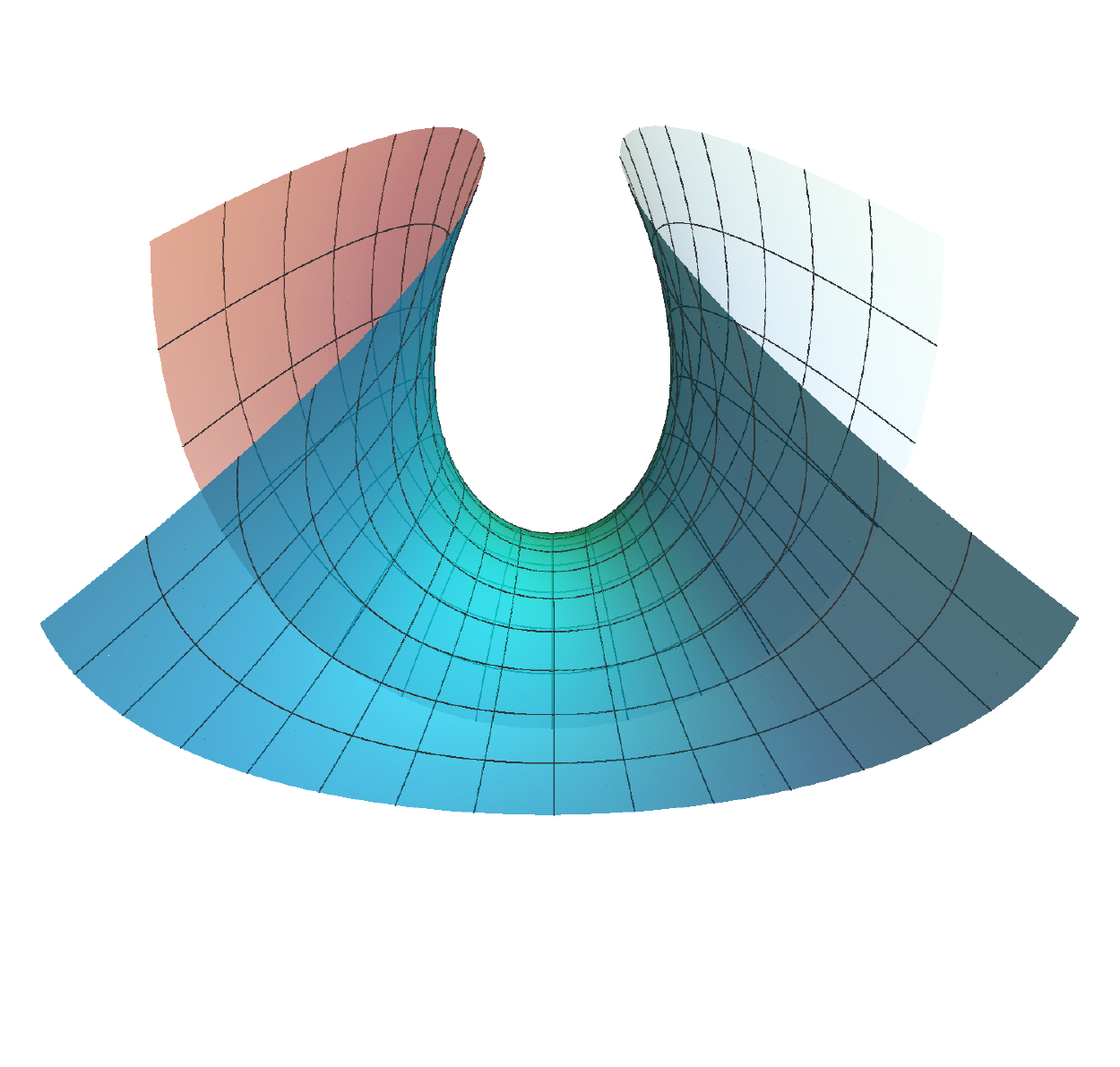}
\vspace{-1cm}
\caption{An $\H$-surface constructed over the curve in Enneper's minimal surface invariant under the reflection with respect to the plane $\{x=0\}$; here $\H(x)=\langle x,e_3\rangle,\ \forall x\in\s2$.}
\label{fig:henneper}
\end{figure}

We can also choose as Björling data the following
$$
\begin{array}{c}
\beta(s)=(-\cos s-1/3\cos(3s),\sin s-1/3\sin(3s),\sin(2s)),\\
B(s)=(\cos s+\cos(3s),2\cos(2s)\sin s,-2\sin(2s)).
\end{array}
$$
Again, for $\H=0$ the surface given by Theorem \ref{thbjorling} is Enneper's minimal surface, and we will call $\beta(s)$ \emph{Enneper's core curve}, see \cite{LoWe}. For the function $\H(x)=\langle x,e_3\rangle,\ \forall x\in\S^2$, the translating soliton arising also resembles to Enneper's minimal surface, see Figure \ref{fig:ennepers}, left. This time we cannot guarantee that the \emph{hole} in the middle will eventually close, as we fail to have an explicit parametrization. If we make the vector $B(s)$ twist along the curve $\beta(s)$ and odd number of times, we get another translating soliton with the topology of a Möbius strip; see Figure \ref{fig:ennepers}, right.

\begin{figure}[H]
\vspace{-2cm}
\hspace{-2cm}
\includegraphics[width=.7\textwidth]{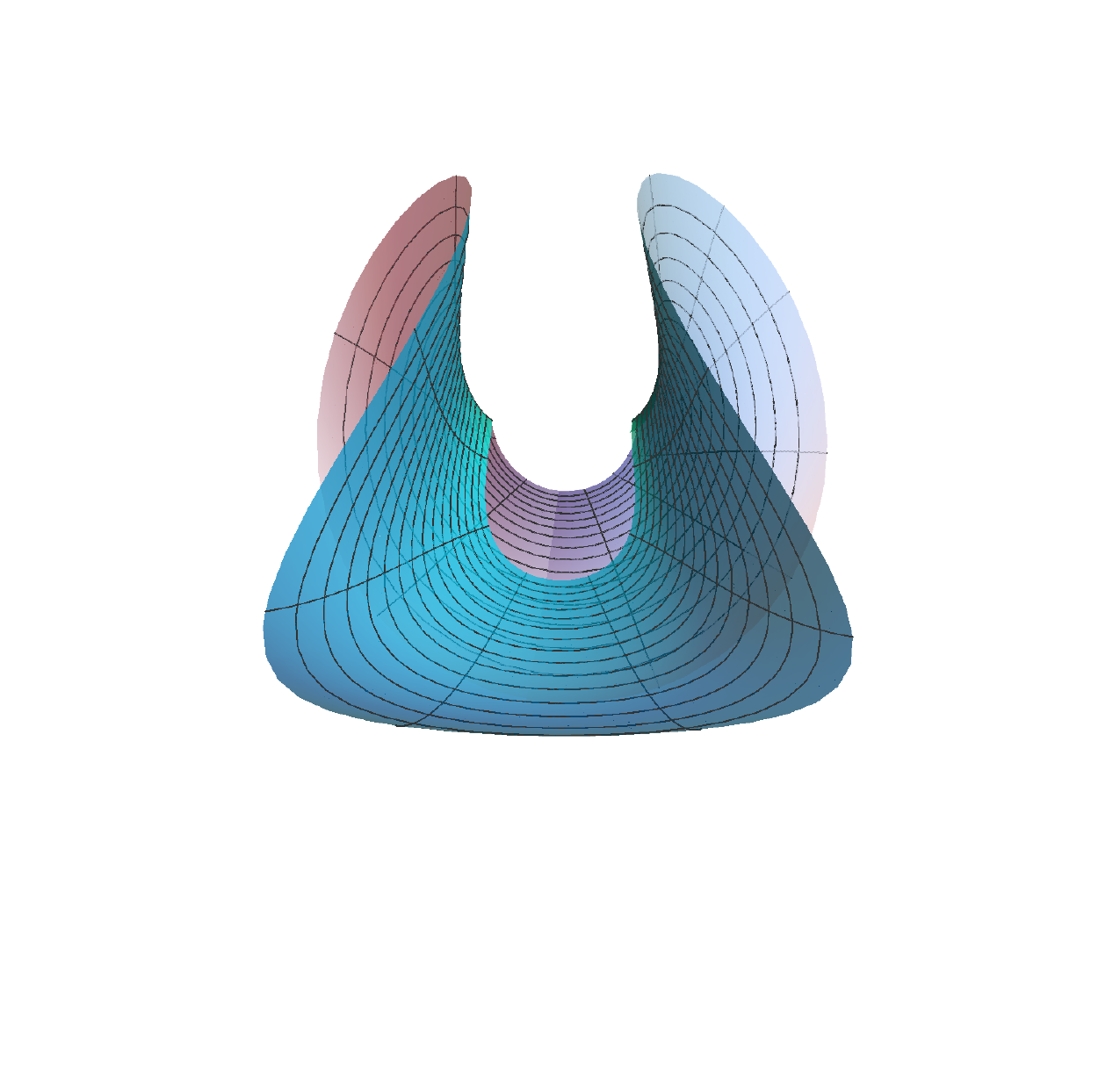}
\hspace{-2cm}
\includegraphics[width=.7\textwidth]{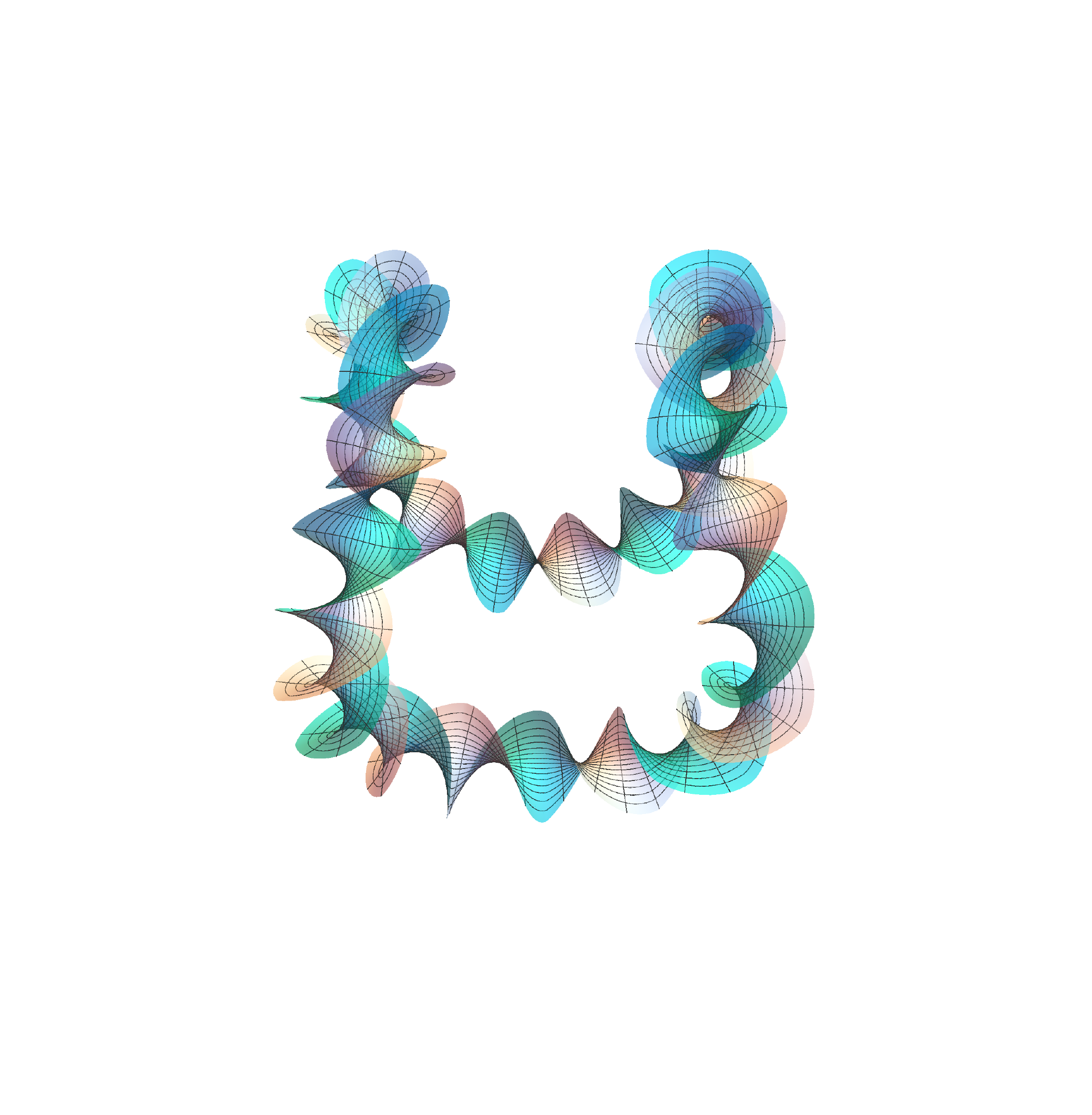}
\vspace{-3cm}
\caption{Two self-translating solitons of the mean curvature flow based on Enneper's core curve.}
\label{fig:ennepers}
\end{figure}

\noindent The author was partially supported by MICINN-FEDER Grant No. MTM2016-80313-P, Junta de Andalucía Grant No. FQM325 and FPI-MINECO Grant No. BES-2014-067663.
 \end{document}